\documentclass[11pt]{amsart}

\usepackage[T1]{fontenc}
\usepackage[utf8]{inputenc}
\usepackage{newtxtext,newtxmath}
\usepackage{microtype}
\usepackage{mathtools}
\usepackage{booktabs}
\usepackage{array}
\usepackage{float}
\usepackage{enumitem}
\usepackage{xcolor}
\usepackage[
  colorlinks=true,
  linkcolor=blue!45!black,
  citecolor=blue!45!black,
  urlcolor=blue!45!black
]{hyperref}
\usepackage[nameinlink,noabbrev,capitalize]{cleveref}

\allowdisplaybreaks[2]
\setlist[enumerate]{leftmargin=2.15em,itemsep=2pt,topsep=4pt}
\setlist[itemize]{leftmargin=2.15em,itemsep=2pt,topsep=4pt}

\newtheorem{theorem}{Theorem}[section]
\newtheorem{proposition}[theorem]{Proposition}
\newtheorem{corollary}[theorem]{Corollary}

\theoremstyle{definition}
\newtheorem{definition}[theorem]{Definition}
\newtheorem{example}[theorem]{Example}
\theoremstyle{remark}
\newtheorem{remark}[theorem]{Remark}

\newcommand{\R}{\mathbb R}
\newcommand{\C}{\mathbb C}
\newcommand{\Z}{\mathbb Z}
\newcommand{\Q}{\mathbb Q}
\newcommand{\Cr}{C_r^*}
\newcommand{\reg}{\mathrm{reg}}
\newcommand{\st}{\mathrm{st}}
\newcommand{\im}{\operatorname{im}}
\newcommand{\indexop}{\operatorname{index}}
\newcommand{\SL}{\operatorname{SL}}
\newcommand{\SO}{\operatorname{SO}}
\newcommand{\SU}{\operatorname{SU}}
\newcommand{\Sp}{\operatorname{Sp}}
\newcommand{\dd}{\,\mathrm d}

\title[L-packet Multiplicity and Integral Structure]
{ L-packet Multiplicity and Integral Structure in the $K$-Theory of Real Inner Forms}
\author{Xinan Dai}
\address{
Key Laboratory for Information Science of Electromagnetic Waves,
College of Future Information and Technology,
Fudan University, Shanghai, China
}
\email{xndai23@m.fudan.edu.cn}
\author{Kuok Fai Chao}
\address{Lui Che Woo College, University of Macau, Macau, China}
\email{kchao@um.edu.mo}
\date{August 3, 2026}

\subjclass[2020]{Primary 46L80, 22E46; Secondary 58J20}
\keywords{real reductive groups, inner forms, discrete series, $L$-packets, orbital integrals, equivariant index, reduced group $C^*$-algebras, character identity}

\hypersetup{
  pdftitle={Stable Packets and Integral Splitting in the K-Theory of Real Inner Forms},
  pdfauthor={Xinan Dai},
  pdfsubject={Stable packet lattices, orbital traces, and integral splitting for real inner forms},
  pdfkeywords={real reductive groups, inner forms, discrete series, L-packets, orbital integrals, K-theory}
}

\begin{document}

\begin{abstract}
Shelstad's discrete-series character identity identifies a finite-dimensional
character of a compact inner form with a signed sum of characters in a L-packet
on an equal-rank real form.  Building on the index-theoretic comparison and
orbital-trace separation theorem of Hochs and Wang, we study the integral
content of this identity in the group
\(K_0(C_r^*G)\).  Two homomorphisms arise naturally:
\[
  \mathcal C_G:K_0(C_r^*G)\longrightarrow R(G_c),
  \qquad
  \mathcal J_G:R(G_c)\longrightarrow K_0(C_r^*G).
\]
The first recognizes a stable elliptic orbital trace as a compact virtual
character; the second sends a compact irreducible representation to the signed
\(K\)-theory class of the corresponding discrete-series L-packet.  We prove
\[
  \mathcal C_G\mathcal J_G=[W_G:W_K]\,\mathrm{id}_{R(G_c)}.
\]
Thus L-packet transfer is integral on the stable L-packet lattice, whereas the
projection from the full \(K\)-group onto that lattice generally requires
inverting the regular L-packet cardinality.  More precisely, if
\(S_G=\operatorname{im}\mathcal J_G\) and
\(U_G=\ker\mathcal C_G\), then
\[
  0\longrightarrow S_G\oplus U_G\longrightarrow K_0(C_r^*G)
  \longrightarrow R(G_c)/[W_G:W_K]R(G_c)\longrightarrow0.
\]
We explain the mechanism first through an elementary lattice model and then
work it out for \(\mathrm{SL}(2,\mathbb R)\) and for symplectic inner forms.  In
rank one the defect is \((\mathbb Z/2\mathbb Z)[s]\); for
\(\mathrm{Sp}(p,q)\) the multiplier is \(\binom{p+q}{p}\), so odd obstruction
primes occur naturally.
\end{abstract}

\maketitle

\section{Introduction}

Let \(G\) be an equal-rank real semisimple group and let \(G_c\) be a compact
inner form.  For a regular Harish--Chandra parameter \(\lambda\), Shelstad's
identity compares the L-packet \(\Pi_\lambda(G)\) with the single compact
representation \(E_\lambda\).  On a common regular elliptic Cartan it has the
form
\begin{equation}\label{eq:intro-shelstad}
  (-1)^{\dim(G/K)/2}
  \sum_{\pi\in\Pi_\lambda(G)}\Theta_\pi(t)
  =\chi_{E_\lambda}(t).
\end{equation}
Here the left-hand side is stable: although the individual discrete-series
characters depend on a choice of member of the L-packet, their signed sum is the
object that persists across inner forms.  Shelstad proved the relevant
character relations for tempered L-packets in \cite[Theorem~6.3]{Shelstad1979};
Hochs and Wang gave an index-theoretic proof of the discrete-series case
\cite[Theorem~2.8]{HochsWang2018Shelstad}.

Equation \eqref{eq:intro-shelstad} is a pointwise equality of character
functions.  Index theory supplies more structure.  A discrete-series
representation determines a class in \(K_0(C_r^*G)\), and its character on the
elliptic set is recovered by orbital traces.  Conversely, the compact
representation \(E_\lambda\) is the equivariant index of a Dolbeault--Dirac
operator on \(G_c/T\).  Hochs and Wang compare this compact index with indices
on \(G/T\), and show that elliptic orbital traces separate
\(K_0(C_r^*G)\) \cite[Corollary~4.1 and Theorem~6.4]{HochsWang2019Orbital}.
These results make it possible to ask an integral question that is not visible
in \eqref{eq:intro-shelstad}:
\begin{quote}
If a compact character is lifted to a L-packet class and then descended again
by stable orbital averaging, do the two operations cancel over \(\mathbb Z\)?
\end{quote}

The answer is no unless the regular L-packet has one member.  The two operations
compose to multiplication by the L-packet cardinality.  The reason is already
visible in rank one.  For \(G=\mathrm{SL}(2,\mathbb R)\), a regular elliptic
element \(t\) and its inverse \(t^{-1}\) lie in the same stable conjugacy class
but in different ordinary conjugacy classes.  A L-packet lift has the same
compact character as its ordinary orbital trace at both points.  Stable
averaging therefore counts that character twice.

The elementary lattice identity behind this phenomenon is
\[
  \mathbb Z^2
  \supset
  \mathbb Z(1,1)\oplus\mathbb Z(1,-1),
  \qquad
  \frac{\mathbb Z^2}
  {\mathbb Z(1,1)\oplus\mathbb Z(1,-1)}\cong\mathbb Z/2\mathbb Z.
\]
Over \(\mathbb Q\), the vector \((1,0)\) splits into its diagonal and
anti-diagonal parts,
\[
  (1,0)=\frac12(1,1)+\frac12(1,-1),
\]
but the factor \(1/2\) is unavoidable over \(\mathbb Z\).  The general theorem
is an infinite-rank and representation-theoretic version of this calculation;
\cref{sec:lattice-model} isolates the model before any analytic machinery is
introduced.

We now state the result.  Let \(G\) be a connected linear real semisimple group
with finite center and a compact Cartan subgroup \(T\).  Let \(K<G\) be a
maximal compact subgroup containing \(T\), and fix a compact inner form
\(G_c\) together with an identification of \(T\) with a maximal torus of
\(G_c\).  Compatible orbital measures are fixed throughout.  Set
\[
  W_G=W(G_{\mathbb C},T_{\mathbb C})\cong N_{G_c}(T)/T,
  \qquad
  W_K=N_K(T)/T,
  \qquad
  m_G=[W_G:W_K].
\]
The integer \(m_G\) is the number of members of every regular discrete-series
L-packet on \(G\).

\begin{theorem}[Main theorem]\label{thm:intro-main}
With the preceding data fixed, there are additive homomorphisms
\[
  \mathcal C_G:K_0(C_r^*G)\longrightarrow R(G_c),
  \qquad
  \mathcal J_G:R(G_c)\longrightarrow K_0(C_r^*G)
\]
with the following properties.
\begin{enumerate}[label=\textup{(\roman*)}]
\item \(\mathcal C_G\) is characterized by stable elliptic orbital traces on a
fixed dense set of regular elements, while \(\mathcal J_G\) is characterized
by ordinary elliptic orbital traces.
\item \(\mathcal C_G\) is surjective, \(\mathcal J_G\) is injective, and
\[
  \mathcal C_G\circ\mathcal J_G
  =m_G\,\mathrm{id}_{R(G_c)}.
\]
\item If
\[
  S_G=\operatorname{im}\mathcal J_G,
  \qquad
  U_G=\ker\mathcal C_G,
\]
then \(S_G\cap U_G=0\) and there is a short exact sequence
\begin{equation}\label{eq:intro-exact}
  0\longrightarrow S_G\oplus U_G
  \longrightarrow K_0(C_r^*G)
  \longrightarrow R(G_c)/m_GR(G_c)
  \longrightarrow0.
\end{equation}
\end{enumerate}
For every regular parameter \(\lambda\) with
\(\lambda-\rho\in X^*(T)\),
\begin{equation}\label{eq:intro-packet}
  \mathcal J_G([E_\lambda])
  =(-1)^{\dim(G/K)/2}
    \sum_{\pi\in\Pi_\lambda(G)}[\pi].
\end{equation}
\end{theorem}

The theorem should be read in two stages.  The analytic stage constructs the
maps.  Stable index comparison shows that the stable orbital trace of any
\(K\)-class is an integral virtual character of \(G_c\); this gives
\(\mathcal C_G\).  Shelstad's identity lifts an irreducible compact character
to the signed class of an entire L-packet; this gives \(\mathcal J_G\).  The
algebraic stage is then short: a compact character is invariant under the
absolute Weyl group, so stable averaging counts it once for every coset in
\(W_G/W_K\).  This gives
\(\mathcal C_G\mathcal J_G=m_G\) and the exact sequence
\eqref{eq:intro-exact}.

This distinction also clarifies the transfer statement.  L-packet transfer
between two inner forms is integral on the lattices \(S_G\): it is obtained by
using the compact representation as a common label.  By contrast, a projector
from the full group \(K_0(C_r^*G)\) onto \(S_G\) is
\[
  P_G^{\mathrm{st}}
  =\frac1{m_G}\mathcal J_G\mathcal C_G,
\]
so it is generally defined only after \(m_G\) has been inverted.  The quotient
in \eqref{eq:intro-exact} records exactly which primes obstruct integral
splitting.  The same integer has three interpretations:
\[
  m_G=|W_G/W_K|=|\Pi_\lambda(G)|=\chi(X^d),
\]
where \(X^d\) is a compact dual of \(G/K\).  The first equality drives the
proof, the second is the L-packet count, and the third is the Hopf--Samelson
Euler-characteristic formula \cite{HopfSamelson1940}.

The examples are part of the argument rather than decoration.  For
\(\mathrm{SL}(2,\mathbb R)\) we write both maps on
\(R(\mathrm{SO}(2))=\mathbb Z[z,z^{-1}]\), recover the trigonometric character
identity, and exhibit a class whose stable projection necessarily contains a
factor \(1/2\).  For symplectic groups we obtain
\[
  m_{\mathrm{Sp}(2n,\mathbb R)}=2^n,
  \qquad
  m_{\mathrm{Sp}(p,q)}=\binom{p+q}{p}.
\]
Thus the split form has a purely \(2\)-primary defect, whereas quaternionic
forms can have odd obstruction primes; \(\mathrm{Sp}(1,2)\) gives the first
such example, with multiplier \(3\).

The scope is deliberately narrower than general endoscopy.  We work with
connected semisimple equal-rank groups of finite center and ordinary elliptic
orbital traces.  Refined inner twists and canonical covers are treated in the
modern construction of real discrete-series L-packets by Adams and Kaletha
\cite[Sections~3--5]{AdamsKaletha2026}.  Beyond the discrete-series range,
higher orbital integrals and cyclic cocycles provide additional functionals on
\(K\)-theory \cite{SongTang2025}.  Huang's spinorial description of certain
transfer factors suggests a relative Dirac interpretation of endoscopic
transfer \cite[Theorem~1.1]{Huang2021}, but no endoscopic
\(KK\)-correspondence is claimed here.

The paper is organized as follows.  \Cref{sec:lattice-model} gives the lattice
model that predicts the exact sequence.  \Cref{sec:inputs} fixes the inner-form
data and records the analytic inputs.  The maps \(\mathcal C_G\) and
\(\mathcal J_G\) are constructed in \cref{sec:maps}.  The multiplier formula,
the integral defect, and the compact-dual interpretation are proved in
\cref{sec:defect}.  Stable transfer is developed in \cref{sec:transfer}.  The
rank-one and symplectic examples occupy \cref{sec:sl2,sec:symplectic}.

\section{A lattice model for the obstruction}\label{sec:lattice-model}

Before turning to orbital integrals, it is useful to separate the algebraic
mechanism from the analytic construction of the maps.  Let \(B\) be a free
abelian group, let \(m\geq1\), and put \(A=B^{\oplus m}\).  Define
\[
  j:B\longrightarrow A,
  \qquad
  j(b)=(b,\ldots,b),
\]
and
\[
  c:A\longrightarrow B,
  \qquad
  c(b_1,\ldots,b_m)=b_1+\cdots+b_m.
\]
The map \(j\) repeats one label in all \(m\) positions, while \(c\) forgets the
positions by summing them.  Hence \(cj=m\,\mathrm{id}_B\).

\begin{proposition}[The elementary defect sequence]\label{prop:lattice-model}
Let \(S=\operatorname{im}j\) and \(U=\ker c\).  Then \(S\cap U=0\), and
\[
  0\longrightarrow S\oplus U\longrightarrow A
  \overset{\bar c}{\longrightarrow}B/mB\longrightarrow0,
  \qquad
  \bar c(a)=c(a)\bmod mB,
\]
is exact.  After inverting \(m\), the projector onto \(S\) along \(U\) is
\[
  p=\frac1m jc.
\]
\end{proposition}

\begin{proof}
If \(j(b)\in U\), then \(mb=cj(b)=0\); freeness of \(B\) gives \(b=0\).
The map \(\bar c\) is surjective.  If \(c(a)=mb\), then
\(a-j(b)\in U\), so \(a\in S+U\).  This identifies the kernel of
\(\bar c\) with \(S\oplus U\).  Finally, \(p^2=p\) follows from
\(cj=m\,\mathrm{id}_B\), and its image and kernel are \(S\) and \(U\).
\end{proof}

For \(B=\mathbb Z\) and \(m=2\), the subgroups are the diagonal and
anti-diagonal lattices in \(\mathbb Z^2\).  Their generators form the matrix
\[
  \begin{pmatrix}1&1\\[2pt]1&-1\end{pmatrix},
\]
whose determinant is \(-2\).  This determinant is the obstruction in its
smallest possible form.  It says simultaneously that the two rational lines
are complementary and that their integral lattices miss one parity class.

The group-theoretic theorem follows the same pattern, but the analytic work is
needed to construct the analogues of \(c\) and \(j\).  The dictionary is:
\begin{table}[H]
\centering
\small
\renewcommand{\arraystretch}{1.15}
\caption{The lattice model and the group-theoretic construction.}
\label{tab:lattice-dictionary}
\begin{tabular}{@{}p{0.27\textwidth}p{0.62\textwidth}@{}}
\toprule
Lattice model & Real-group construction\\
\midrule
\(B\) & \(R(G_c)\), with basis the irreducible compact representations\\
\(A\) & \(K_0(C_r^*G)\), in the formal role of the ambient lattice\\
\(j\) & the L-packet lift \(\mathcal J_G\)\\
\(c\) & stable trace descent \(\mathcal C_G\)\\
\(m\) & \([W_G:W_K]\), the regular L-packet cardinality\\
\bottomrule
\end{tabular}
\end{table}

The table is a guide, not an assertion that
\(K_0(C_r^*G)\cong R(G_c)^{\oplus m_G}\).  What survives from the model is the
pair of trace-characterized maps and the relation
\(\mathcal C_G\mathcal J_G=m_G\).  Once that relation is proved, the integral
defect sequence is the same elementary argument as in
\cref{prop:lattice-model}.

\section{Inner data and analytic inputs}\label{sec:inputs}

\subsection{Fixed data and L-packet notation}

Throughout the paper, $G$ is a connected linear real semisimple group with finite center, $K<G$ is a maximal compact subgroup, and $T\subset K$ is a compact Cartan subgroup of $G$.  We fix a compact inner form $G_c$ and an inner identification under which $T$ is also a maximal torus of $G_c$.  All orbital measures are chosen compatibly with the conventions in \cite[Sections~2 and~6]{HochsWang2019Orbital}.  The constructions are canonical relative to these data.  We write $R(G_c)$ for the Grothendieck group of finite-dimensional complex representations of $G_c$, $\widehat{G_c}$ for its irreducible classes, and $\Cr G$ for the reduced group $C^*$-algebra.  A representation and its class in $R(G_c)$ will usually be denoted by the same symbol.

Let
\[
  W_G=W(G_{\C},T_{\C})\cong N_{G_c}(T)/T,
  \qquad
  W_K=N_K(T)/T.
\]
Choose a positive root system
\[
  R^+\subset R(\mathfrak g_{\C},\mathfrak t_{\C})
\]
and let $\rho$ be its half-sum.  Put
\[
  m_G=[W_G:W_K],
  \qquad
  q(G)=\frac12\dim(G/K),
  \qquad
  \varepsilon_G=(-1)^{q(G)}.
\]
The positive system is used to write down the Dolbeault operators and label the L-packets.  The maps constructed below are ultimately characterized by orbital traces, so they do not depend on this auxiliary labeling.

The main symbols are collected here for reference.
\begin{table}[H]
\centering
\small
\renewcommand{\arraystretch}{1.15}
\caption{Notation used throughout the paper.}
\label{tab:notation}
\begin{tabular}{@{}p{0.24\textwidth}p{0.65\textwidth}@{}}
\toprule
Symbol & Meaning\\
\midrule
\(T_{\mathrm{reg}}\) & regular elements of the common compact Cartan \(T\)\\
\(T^\dagger\) & a fixed Weyl-invariant dense subset on which the fixed-point formulas hold\\
\(Q_G(\nu)\) & the noncompact Dolbeault--Dirac index on \(G/T\) with weight \(\nu\)\\
\(Q_c(\nu)\) & the compact equivariant index on \(G_c/T\), an element of \(R(G_c)\)\\
\(\tau_t^G\), \(\tau_t^{\mathrm{st},G}\) & ordinary and stable elliptic orbital traces\\
\(\mathcal C_G\), \(\mathcal J_G\) & stable trace descent and L-packet lift\\
\(S_G\), \(U_G\) & \(\operatorname{im}\mathcal J_G\) and \(\ker\mathcal C_G\)\\
\bottomrule
\end{tabular}
\end{table}

Let $\lambda\in i\mathfrak t^*$ be regular and dominant for $R^+$, and suppose that $\lambda-\rho$ is a character of $T$.  The discrete-series representations with infinitesimal character $\lambda$ form the L-packet
\begin{equation}\label{eq:packet-cosets}
  \Pi_\lambda(G)
  =\{\pi_{w^{-1}\lambda}:[w]\in W_G/W_K\},
  \qquad
  |\Pi_\lambda(G)|=m_G;
\end{equation}
see \cite[Proposition~2.7]{HochsWang2018Shelstad}.  The compact form has a single corresponding representation, denoted $E_\lambda$, of highest weight $\lambda-\rho$.  Conversely, every irreducible representation of $G_c$ arises this way after $R^+$ has been fixed, by the Borel--Weil--Bott theorem \cite{Bott1957}.

\begin{example}[The compact case]\label{ex:compact-case}
If \(G\) is already compact, then \(K=G=G_c\), \(W_K=W_G\), and
\(m_G=1\).  Under the standard identification
\(K_0(C_r^*G)\cong R(G)\), both maps constructed below are the identity.
Thus \(S_G=K_0(C_r^*G)\), \(U_G=0\), and there is no integral defect.  The
obstruction begins precisely when a stable elliptic class splits into more than
one ordinary conjugacy class.
\end{example}

\subsection{Orbital traces}

For \(t\in T_{\reg}\) and a Schwartz function \(f\) on \(G\), define
\[
  \tau_t^G(f)
  =\int_{G/Z_G(t)} f(xtx^{-1})\,\dd(xZ_G(t)).
\]
This integral averages \(f\) over the ordinary conjugacy class of \(t\).  It is
a trace on the Harish--Chandra Schwartz algebra and therefore pairs with
\(K\)-theory:
\[
  \tau_t^G:K_0(\Cr G)\longrightarrow\C;
\]
see \cite[Sections~2--3]{HochsWang2019Orbital}.  For the natural class of a
discrete-series representation, this pairing recovers the Harish--Chandra
character.  Thus orbital traces are the bridge between the geometric
\(K\)-classes and the character identity.

Stable conjugacy is coarser than ordinary conjugacy.  On the common elliptic
Cartan, ordinary conjugacy is measured by \(W_K\), while stable conjugacy is
measured by the larger group \(W_G\).  Consequently the stable orbital trace is
\begin{equation}\label{eq:stable-orbital-sum}
  \tau_t^{\st,G}
  =\sum_{[w]\in W_G/W_K}\tau_{wtw^{-1}}^G,
  \qquad t\in T_{\reg}.
\end{equation}
Indeed, ordinary conjugacy on $T_{\reg}$ is controlled by $W_K$, whereas stable conjugacy is controlled by $W_G$; see Arthur \cite[Section~27, p.~194]{Arthur2005} and Hochs--Wang \cite[Lemma~6.10]{HochsWang2019Orbital}.

For \(\mathrm{SL}(2,\mathbb R)\), one has \(W_K=1\) and
\(W_G\cong\mathbb Z/2\mathbb Z\); hence
\(\tau_t^{\mathrm{st},G}=\tau_t^G+\tau_{t^{-1}}^G\).  This two-term formula is
the geometric source of the factor \(2\) in \cref{sec:sl2}.  No endoscopic
transfer factor is present in \eqref{eq:stable-orbital-sum}: for regular
elliptic elements in this inner-form setting, stable averaging is literally the
sum over the ordinary classes in one stable class.

The fixed-point formulas in \cite[Theorem~3.2 and Section~6]{HochsWang2019Orbital} hold on a full-measure, hence dense, subset of $T$.  We fix once and for all a $W_G$-invariant dense subset
\[
  T^\dagger\subset T_{\reg}
\]
on which all formulas used below hold.  Such a set is obtained by intersecting the relevant full-measure set with its finitely many $W_G$-translates.  Keeping \(T^\dagger\) explicit avoids an unwarranted pointwise regularity assertion.  Nothing in the argument requires a formula at the exceptional regular elements: compact virtual characters are continuous, and the orbital traces on a dense subset already separate \(K_0(\Cr G)\).  This is why an almost-everywhere fixed-point formula is sufficient for an integral statement about the whole \(K\)-group.

The first analytic input is the following separation theorem.

\begin{theorem}[Hochs--Wang]\label{thm:separation}
If $x\in K_0(\Cr G)$ and $\tau_t^G(x)=0$ for every $t$ in a dense subset of $T_{\reg}$, then $x=0$.
\end{theorem}

\begin{proof}
This is \cite[Corollary~4.1]{HochsWang2019Orbital}.  Its proof uses the orbital-trace formula for Dirac induction and the Connes--Kasparov isomorphism.  The latter was proved for almost connected groups by Chabert, Echterhoff, and Nest \cite{ChabertEchterhoffNest2003}.
\end{proof}

\subsection{Dolbeault--Dirac indices}

For an integral weight $\nu\in i\mathfrak t^*$, let
\[
  L_\nu^G=G\times_T\C_\nu\longrightarrow G/T,
  \qquad
  L_\nu^c=G_c\times_T\C_\nu\longrightarrow G_c/T.
\]
The positive roots determine invariant complex structures, and we write
\[
  D_\nu^G=\bar\partial_{L_\nu^G}+\bar\partial_{L_\nu^G}^*,
  \qquad
  D_\nu^c=\bar\partial_{L_\nu^c}+\bar\partial_{L_\nu^c}^*.
\]
Set
\[
  Q_G(\nu)=\indexop_G(D_\nu^G)\in K_0(\Cr G),
  \qquad
  Q_c(\nu)=\indexop_{G_c}(D_\nu^c)\in R(G_c).
\]
The notation emphasizes that the same weight \(\nu\) is used on both sides:
\[
  (G/T,L_\nu^G)\qquad\longleftrightarrow\qquad(G_c/T,L_\nu^c).
\]
The left index is analytic and belongs to the reduced group \(C^*\)-algebra;
the right index is compact and hence is an honest virtual representation.  By
Borel--Weil--Bott, \(Q_c(\nu)\) is either zero or a signed irreducible
representation after the usual Weyl correction.  The comparison theorem says
that stable orbital traces erase the difference between the two geometries.
The realization of discrete series by Dirac-type operators goes back to Parthasarathy \cite{Parthasarathy1972} and Atiyah--Schmid \cite{AtiyahSchmid1977,AtiyahSchmid1979}.  We use the following precise consequences of the later $C^*$-algebraic theory.

\begin{proposition}[Index-theoretic inputs]\label{prop:index-inputs}
For integral weights $\nu$, the following hold.
\begin{enumerate}[label=\textup{(\roman*)}]
\item The classes $Q_G(\nu)$ generate $K_0(\Cr G)$ as an abelian group.
\item For $w\in W_G$ and $t\in T^\dagger$,
\begin{equation}\label{eq:weyl-covariance}
  \tau_{wtw^{-1}}^G\bigl(Q_G(\nu)\bigr)
  =\tau_t^G\bigl(Q_G(w^{-1}\nu)\bigr).
\end{equation}
\item For $t\in T^\dagger$,
\begin{equation}\label{eq:index-comparison}
  \chi_{Q_c(\nu)}(t)
  =\sum_{[w]\in W_G/W_K}
     \tau_t^G\bigl(Q_G(w^{-1}\nu)\bigr).
\end{equation}
\end{enumerate}
\end{proposition}

\begin{proof}
For part~(i), combine the surjectivity of Dirac induction with its Dolbeault realization in \cite[Proposition~6.5]{HochsWang2019Orbital}; each irreducible Dirac-induction generator is represented by one of the displayed indices.  Part~(ii) is \cite[Lemma~6.6]{HochsWang2019Orbital}.  Part~(iii) is \cite[Theorem~6.4]{HochsWang2019Orbital}, whose original form is \cite[Equation~(3.6)]{HochsWang2018Shelstad}.
\end{proof}

Combining \eqref{eq:stable-orbital-sum}, \eqref{eq:weyl-covariance}, and \eqref{eq:index-comparison} gives the stable identity that drives the construction.

\begin{corollary}[Stable index identity]\label{cor:stable-index}
For every integral $\nu$ and every $t\in T^\dagger$,
\begin{equation}\label{eq:stable-index}
  \tau_t^{\st,G}\bigl(Q_G(\nu)\bigr)
  =\chi_{Q_c(\nu)}(t).
\end{equation}
\end{corollary}

\begin{proof}
The left-hand side is
\[
  \sum_{[w]\in W_G/W_K}
  \tau_{wtw^{-1}}^G\bigl(Q_G(\nu)\bigr).
\]
Use \eqref{eq:weyl-covariance} term by term, and then apply \eqref{eq:index-comparison}.
\end{proof}

For a discrete-series representation $\pi$, let $[\pi]\in K_0(\Cr G)$ denote its natural $K$-theory class.  If $v$ is a unit $K$-finite vector and $d_\pi$ is the formal degree, this class is represented by the convolution idempotent $d_\pi m_{v,v}$; see the proof of \cite[Corollary~4.10]{HochsWang2019Orbital}.  With the compatible measure normalization,
\begin{equation}\label{eq:natural-class-character}
  \tau_t^G([\pi])=\Theta_\pi(t),
  \qquad t\in T_{\reg};
\end{equation}
see \cite[Proposition~2.11]{HochsWang2018Shelstad}.  In particular,
\begin{equation}\label{eq:natural-class-index}
  [\pi_\lambda]=\varepsilon_G Q_G(\lambda-\rho).
\end{equation}
The compact-inner-form character identity reads
\begin{equation}\label{eq:shelstad-compact}
  \varepsilon_G\sum_{\pi\in\Pi_\lambda(G)}\Theta_\pi(t)
  =\chi_{E_\lambda}(t),
  \qquad t\in T_{\reg};
\end{equation}
see \cite[Theorem~2.8]{HochsWang2018Shelstad}.

\section{Two maps selected by orbital traces}\label{sec:maps}

The two maps point in opposite directions, but they solve different problems.  Starting with a noncompact \(K\)-class, stable trace descent asks whether its stable orbital trace comes from a compact virtual character.  Starting with a compact representation, the L-packet lift asks which noncompact \(K\)-class has that character as its ordinary orbital trace.  The first question requires the stable index identity; the second requires Shelstad's character identity.

A useful summary is
\begin{equation}\label{eq:two-map-summary}
  R(G_c)
  \overset{\mathcal J_G}{\longrightarrow}
  K_0(\Cr G)
  \overset{\mathcal C_G}{\longrightarrow}
  R(G_c),
  \qquad
  E\longmapsto
  \varepsilon_G\!\sum_{\pi\in\Pi(E)}[\pi]
  \longmapsto m_GE.
\end{equation}
The last arrow is the conclusion of \cref{thm:CJ}; the present section constructs the first two arrows without assuming that formula.

\subsection{Stable trace descent}

A tempting definition would be to send \(x\in K_0(\Cr G)\) to the function
\(t\mapsto\tau_t^{\mathrm{st},G}(x)\).  The issue is that an invariant function
on \(T\) need not be an integral virtual character.  The stable index identity
resolves exactly this point: on the Dolbeault generators, the function is the
character of \(Q_c(\nu)\), and all relations among the noncompact generators
remain relations among the compact indices.

\begin{theorem}[Stable trace descent]\label{thm:C-map}
There is a unique additive homomorphism
\[
  \mathcal C_G:K_0(\Cr G)\longrightarrow R(G_c)
\]
such that
\begin{equation}\label{eq:C-characterization}
  \chi_{\mathcal C_G(x)}(t)=\tau_t^{\st,G}(x)
  \qquad(x\in K_0(\Cr G),\ t\in T^\dagger).
\end{equation}
It satisfies
\begin{equation}\label{eq:C-on-index}
  \mathcal C_G\bigl(Q_G(\nu)\bigr)=Q_c(\nu)
\end{equation}
for every integral $\nu$, and it is surjective.
\end{theorem}

\begin{proof}
By \cref{prop:index-inputs}(i), write
\[
  x=\sum_{j=1}^r a_jQ_G(\nu_j),
  \qquad a_j\in\Z,
\]
and set
\[
  \mathcal C_G(x)=\sum_{j=1}^r a_jQ_c(\nu_j).
\]
Suppose that the displayed expression for $x$ is zero.  Its stable orbital trace then vanishes on $T^\dagger$.  By \cref{cor:stable-index}, the compact virtual representation on the right has character zero on the dense set $T^\dagger$.  A compact virtual character is continuous and is determined by its restriction to a dense subset of a maximal torus, so the virtual representation is zero.  Thus the definition is independent of the chosen expression.  The same argument proves \eqref{eq:C-characterization} and uniqueness.

For surjectivity, the Borel--Weil--Bott theorem says that the classes $Q_c(\nu)$ generate $R(G_c)$ as $\nu$ ranges over the integral weights; see \cite{Bott1957}.  Equation \eqref{eq:C-on-index} supplies preimages for these generators.
\end{proof}

The right-hand side of \eqref{eq:C-characterization} is initially only a stable orbital-trace function.  The theorem says that it is not merely invariant: on \(T^\dagger\) it is the restriction of an integral virtual character of \(G_c\).  In particular,
\[
  Q_G(\nu)\longmapsto Q_c(\nu)
\]
is compatible with every integral relation among the noncompact indices.  This is the substantive point: without it, one could define stable descent only after passing to a space of functions, not as a homomorphism into the representation ring.

\subsection{The L-packet lift}

There is generally no preferred representation inside a nontrivial L-packet.
Choosing one would destroy stability and would not be compatible with inner
forms.  The signed sum of all L-packet members is instead singled out by
\eqref{eq:shelstad-compact}: its ordinary elliptic orbital trace is exactly the
compact character.  This trace property, rather than a choice of parameter
representative, is what makes the lift intrinsic.

\begin{definition}[L-packet lift]\label{def:J-map}
For $E_\lambda\in\widehat{G_c}$, define
\begin{equation}\label{eq:J-definition}
  \mathcal J_G([E_\lambda])
  =\varepsilon_G\sum_{\pi\in\Pi_\lambda(G)}[\pi],
\end{equation}
and extend linearly to
\[
  \mathcal J_G:R(G_c)\longrightarrow K_0(\Cr G).
\]
\end{definition}

\begin{proposition}[Orbital characterization]\label{prop:J-characterization}
For $E\in R(G_c)$ and $t\in T_{\reg}$,
\begin{equation}\label{eq:J-characterization}
  \tau_t^G\bigl(\mathcal J_G(E)\bigr)=\chi_E(t).
\end{equation}
Consequently, $\mathcal J_G$ is injective.  It is the unique additive homomorphism satisfying \eqref{eq:J-characterization} on any dense subset of $T_{\reg}$.
\end{proposition}

\begin{proof}
For an irreducible $E_\lambda$, equations \eqref{eq:J-definition}, \eqref{eq:natural-class-character}, and \eqref{eq:shelstad-compact} give \eqref{eq:J-characterization}; the general case follows by linearity.  If $\mathcal J_G(E)=0$, then $\chi_E$ vanishes on $T_{\reg}$, so $E=0$.  If two homomorphisms have the stated trace property on a dense subset, their values differ by a $K$-class whose elliptic orbital traces vanish there.  \Cref{thm:separation} then gives uniqueness.
\end{proof}

Combining \cref{thm:C-map,prop:J-characterization}, the two maps can be read directly from traces:
\[
  \chi_{\mathcal C_G(x)}=\tau^{\mathrm{st},G}(x),
  \qquad
  \tau^G(\mathcal J_G(E))=\chi_E.
\]
The first formula averages before recognizing a compact character; the second
chooses the unique \(K\)-class whose ordinary traces already equal that
character.  The order of these operations is the source of the multiplier.

\subsection{The two trace-selected lattices}

Define
\begin{equation}\label{eq:S-U-definition}
  S_G=\im\mathcal J_G,
  \qquad
  U_G=\ker\mathcal C_G.
\end{equation}
We call $S_G$ the \emph{stable L-packet lattice} and $U_G$ the \emph{stable-orbital kernel}.  The terminology records two different trace conditions.  A class in $U_G$ need not have zero ordinary orbital traces; its traces cancel only after stable summation.

\begin{proposition}[Intrinsic descriptions]\label{prop:intrinsic}
The subgroups in \eqref{eq:S-U-definition} satisfy
\begin{align}
  S_G
  &=\left\{x\in K_0(\Cr G):
      \tau_t^G(x)=\chi_E(t)\text{ on }T^\dagger
      \text{ for some }E\in R(G_c)\right\},\label{eq:S-intrinsic}\\
  U_G
  &=\left\{x\in K_0(\Cr G):
      \tau_t^{\st,G}(x)=0\text{ for every }t\in T^\dagger\right\}.
      \label{eq:U-intrinsic}
\end{align}
\end{proposition}

\begin{proof}
The description of $U_G$ is \eqref{eq:C-characterization}.  One inclusion in \eqref{eq:S-intrinsic} follows from \eqref{eq:J-characterization}.  Conversely, if $\tau_t^G(x)=\chi_E(t)$ on $T^\dagger$, then $x$ and $\mathcal J_G(E)$ have the same orbital traces on a dense subset.  Apply \cref{thm:separation}.
\end{proof}

\begin{remark}[Dependence on choices]\label{rem:choices}
The positive roots enter the operators $Q_G(\nu)$ and the labeling of $E_\lambda$ and $\Pi_\lambda(G)$.  The characterizations \eqref{eq:C-characterization} and \eqref{eq:J-characterization} show that the resulting maps do not depend on that labeling.  They do depend on the fixed identification of the elliptic Cartan data and on compatible choices of orbital measures.  We use ``canonical'' only relative to those choices.
\end{remark}

\section{L-packet multiplicity and the integral defect}\label{sec:defect}

We now close the two constructions into a single calculation.  At this point the analytic work is finished: the proof is the same diagonal-versus-sum calculation as in \cref{sec:lattice-model}, with Weyl cosets replacing the \(m\) coordinates.

\begin{theorem}[L-packet multiplicity formula]\label{thm:CJ}
For every $E\in R(G_c)$,
\begin{equation}\label{eq:CJ}
  \mathcal C_G\mathcal J_G(E)=m_GE.
\end{equation}
\end{theorem}

\begin{proof}
Let $t\in T^\dagger$.  Using \eqref{eq:C-characterization}, \eqref{eq:stable-orbital-sum}, and \eqref{eq:J-characterization}, we obtain
\begin{align*}
  \chi_{\mathcal C_G\mathcal J_G(E)}(t)
  &=\tau_t^{\st,G}\bigl(\mathcal J_G(E)\bigr)\\
  &=\sum_{[w]\in W_G/W_K}
    \tau_{wtw^{-1}}^G\bigl(\mathcal J_G(E)\bigr)\\
  &=\sum_{[w]\in W_G/W_K}\chi_E(wtw^{-1}).
\end{align*}
The character $\chi_E$ is $W_G$-invariant on $T$, so the last expression is $m_G\chi_E(t)$.  Equality on the dense set $T^\dagger$ implies equality in $R(G_c)$.
\end{proof}

The formula has a simple interpretation.  The L-packet lift replaces one compact character by a \(K\)-class having the same ordinary elliptic trace.  Stable descent then counts that character once for every ordinary conjugacy class inside its stable class.  The L-packet size is therefore forced by the geometry of conjugacy; it is not a normalization that can be removed while preserving both trace characterizations.

One might try to replace \(\mathcal C_G\) by \(m_G^{-1}\mathcal C_G\).  This produces an inverse to \(\mathcal J_G\) on the stable L-packet lattice, but it is generally not integral on the ambient \(K\)-group.  The next theorem measures the failure precisely rather than hiding it in a normalization.

\begin{theorem}[Integral defect sequence]\label{thm:defect}
The intersection $S_G\cap U_G$ is zero.  Moreover,
\[
  \overline{\mathcal C}_G:K_0(\Cr G)\longrightarrow R(G_c)/m_GR(G_c),
  \qquad
  x\longmapsto\mathcal C_G(x)\bmod m_GR(G_c),
\]
fits into a short exact sequence
\begin{equation}\label{eq:defect-sequence}
  0\longrightarrow S_G\oplus U_G
  \longrightarrow K_0(\Cr G)
  \overset{\overline{\mathcal C}_G}{\longrightarrow}
  R(G_c)/m_GR(G_c)
  \longrightarrow0.
\end{equation}
\end{theorem}

\begin{proof}
If $\mathcal J_G(E)\in U_G$, then \cref{thm:CJ} gives $m_GE=0$.  The group $R(G_c)$ is free abelian on the irreducible representations of $G_c$, hence torsion-free; therefore $E=0$.  This proves $S_G\cap U_G=0$.

The map $\overline{\mathcal C}_G$ is surjective because $\mathcal C_G$ is.  Both $U_G$ and $S_G$ lie in its kernel, the latter by \cref{thm:CJ}.  Conversely, if $\mathcal C_G(x)=m_GE$, then
\[
  \mathcal C_G\bigl(x-\mathcal J_G(E)\bigr)=0,
\]
so $x-\mathcal J_G(E)\in U_G$.  Hence $x\in S_G+U_G$, and the sum is direct by the first paragraph.
\end{proof}

Because $R(G_c)$ is free on $\widehat{G_c}$, the cokernel has the explicit form
\begin{equation}\label{eq:defect-direct-sum}
  R(G_c)/m_GR(G_c)
  \cong\bigoplus_{E\in\widehat{G_c}}\Z/m_G\Z.
\end{equation}
Thus each compact irreducible contributes one copy of the same finite defect.  Along a single irreducible direction, the situation is exactly the lattice model with \(B=\mathbb Z\): \(S_G\) plays the role of the diagonal lattice, \(U_G\) the sum-zero lattice, and one residue class modulo \(m_G\) remains.  The statement concerns these two subgroups selected by orbital traces.  It does not rule out unrelated abstract complements in the ambient group \(K_0(\Cr G)\).

\begin{corollary}[Stable projector and obstruction primes]\label{cor:projector}
After inverting $m_G$,
\begin{equation}\label{eq:localized-splitting}
  K_0(\Cr G)\otimes\Z[1/m_G]
  =\bigl(S_G\otimes\Z[1/m_G]\bigr)
   \oplus
   \bigl(U_G\otimes\Z[1/m_G]\bigr).
\end{equation}
The projector onto the first summand is
\begin{equation}\label{eq:stable-projector}
  P_G^{\st}=\frac1{m_G}\mathcal J_G\mathcal C_G.
\end{equation}
It is the unique rational projector with image $S_G\otimes\Q$ and kernel $U_G\otimes\Q$.

For a prime $\ell$, the corresponding decomposition over $\Z_{(\ell)}$ holds if and only if $\ell\nmid m_G$.  In particular, if $m_G>1$, no integral projector has image $S_G$ and kernel $U_G$.
\end{corollary}

\begin{proof}
By \cref{thm:CJ},
\[
  (P_G^{\st})^2
  =\frac1{m_G^2}\mathcal J_G
    (\mathcal C_G\mathcal J_G)\mathcal C_G
  =P_G^{\st}.
\]
It is the identity on $S_G$ and zero on $U_G$.  Exactness of \eqref{eq:defect-sequence} gives \eqref{eq:localized-splitting}.

If $P$ is any rational projector with this image and kernel, then for $x\in K_0(\Cr G)\otimes\Q$ write $P(x)=\mathcal J_G(E)$ with $E\in R(G_c)\otimes\Q$.  Since $x-P(x)\in U_G$,
\[
  \mathcal C_G(x)=\mathcal C_G(P(x))=m_GE,
\]
so $P(x)=m_G^{-1}\mathcal J_G\mathcal C_G(x)$.  This proves uniqueness.

Localizing \eqref{eq:defect-sequence} at $\ell$ shows that the two lattices exhaust the local $K$-group exactly when multiplication by $m_G$ is invertible in $\Z_{(\ell)}$, namely when $\ell\nmid m_G$.  If $\ell\mid m_G$, the localized quotient of the nonzero free group $R(G_c)$ remains nonzero.
\end{proof}

\subsection{The compact-dual Euler characteristic}

Let
\[
  \mathfrak g=\mathfrak k\oplus\mathfrak p
\]
be the Cartan decomposition, and let
\[
  \mathfrak g^d=\mathfrak k\oplus i\mathfrak p
\]
be its compact dual Lie algebra.  Choose compact connected groups $K^d\subset G^d$ with Lie algebras $\mathfrak k\subset\mathfrak g^d$ and a common maximal torus, again denoted $T$; one may pass to a compatible finite central cover to arrange these data.  Put $X^d=G^d/K^d$.  Finite central isogenies do not change the Weyl groups, so the Weyl groups of $(G^d,T)$ and $(K^d,T)$ are $W_G$ and $W_K$.  The Hopf--Samelson formula therefore gives
\begin{equation}\label{eq:euler-characteristic}
  \chi(X^d)=\frac{|W_G|}{|W_K|}=m_G;
\end{equation}
see Hopf--Samelson \cite{HopfSamelson1940}; the compact-dual construction is reviewed in \cite[Chapters~V and~X]{Helgason2001}.

Equation \eqref{eq:euler-characteristic} is not used in the proof of \cref{thm:CJ}.  It explains why the same integer is robust: L-packet cardinality, the number of Weyl fixed-point blocks in the index comparison, and the Euler characteristic of the compact dual all count the same Weyl quotient.  Consequently, the primes obstructing the stable integral splitting are exactly the primes dividing this Euler characteristic.

\section{Stable transfer between inner forms}\label{sec:transfer}

Let \(G_1\) and \(G_2\) be equal-rank real forms in the same inner class.  Fix compatible compact Cartan identifications, orbital measures, and a common compact form \(G_c\).  There are two superficially similar ways to return from \(K_0(\Cr G_i)\) to compact data.  The map \(\mathcal C_i\) is stable averaging and multiplies a L-packet lift by \(m_i\); the inverse \(\mathcal J_i^{-1}\) is defined integrally only on the L-packet lattice \(S_i\).  Integral transfer therefore uses \(\mathcal J_i^{-1}\), not \(m_i^{-1}\mathcal C_i\) on the whole \(K\)-group.  Write
\[
  \mathcal C_i=\mathcal C_{G_i},\quad
  \mathcal J_i=\mathcal J_{G_i},\quad
  S_i=S_{G_i},\quad
  U_i=U_{G_i},\quad
  m_i=[W_G:W_{K_i}].
\]

\begin{theorem}[Integral stable transfer]\label{thm:stable-transfer}
There is a unique isomorphism
\[
  \mathcal T^{\st}_{1\to2}:S_1\overset{\sim}{\longrightarrow}S_2
\]
satisfying
\begin{equation}\label{eq:stable-transfer-definition}
  \mathcal T^{\st}_{1\to2}\bigl(\mathcal J_1(E)\bigr)
  =\mathcal J_2(E)
  \qquad(E\in R(G_c)).
\end{equation}
For a regular parameter $\lambda$,
\begin{equation}\label{eq:packet-transfer}
  \mathcal T^{\st}_{1\to2}
  \left(
    \varepsilon_{G_1}\sum_{\pi\in\Pi_\lambda(G_1)}[\pi]
  \right)
  =
    \varepsilon_{G_2}\sum_{\pi'\in\Pi_\lambda(G_2)}[\pi'].
\end{equation}
For three compatible inner forms,
\begin{equation}\label{eq:transfer-composition}
  \mathcal T^{\st}_{2\to3}\mathcal T^{\st}_{1\to2}
  =\mathcal T^{\st}_{1\to3}.
\end{equation}
\end{theorem}

\begin{proof}
Each $\mathcal J_i$ is an isomorphism from $R(G_c)$ onto $S_i$.  Define
\[
  \mathcal T^{\st}_{1\to2}=\mathcal J_2\mathcal J_1^{-1}.
\]
The L-packet formula and the composition law are immediate.
\end{proof}

If $t_1\in T\subset G_1$ and $t_2\in T\subset G_2$ are corresponding regular elements under the fixed Cartan identification, then \eqref{eq:J-characterization} gives
\begin{equation}\label{eq:transfer-trace}
  \tau_{t_2}^{G_2}\bigl(\mathcal T^{\st}_{1\to2}(x)\bigr)
  =\tau_{t_1}^{G_1}(x),
  \qquad x\in S_1.
\end{equation}
Thus \cref{thm:stable-transfer} is characterized by ordinary elliptic orbital traces on the stable L-packet lattice.  It transfers the stable object as a whole; it does not choose a preferred member of a L-packet.

There is a useful extension to the ambient rational $K$-groups.

\begin{proposition}[Rational extension]\label{prop:rational-transfer}
Define
\begin{equation}\label{eq:rational-transfer}
  \widetilde{\mathcal T}_{1\to2}
  =\frac1{m_1}\mathcal J_2\mathcal C_1:
  K_0(\Cr G_1)\otimes\Q
  \longrightarrow
  K_0(\Cr G_2)\otimes\Q.
\end{equation}
Then $\widetilde{\mathcal T}_{1\to2}$ agrees with $\mathcal T^{\st}_{1\to2}$ on $S_1$, vanishes on $U_1$, and satisfies
\[
  \widetilde{\mathcal T}_{2\to3}
  \widetilde{\mathcal T}_{1\to2}
  =\widetilde{\mathcal T}_{1\to3}.
\]
Moreover,
\[
  \widetilde{\mathcal T}_{G\to G}=P_G^{\st}.
\]
\end{proposition}

\begin{proof}
For $E\in R(G_c)$,
\[
  \widetilde{\mathcal T}_{1\to2}\mathcal J_1(E)
  =\frac1{m_1}\mathcal J_2\mathcal C_1\mathcal J_1(E)
  =\mathcal J_2(E)
\]
by \cref{thm:CJ}; vanishing on $U_1$ is immediate.  For composition,
\begin{align*}
  \widetilde{\mathcal T}_{2\to3}
  \widetilde{\mathcal T}_{1\to2}
  &=\frac1{m_1m_2}
    \mathcal J_3\mathcal C_2\mathcal J_2\mathcal C_1\\
  &=\frac1{m_1}\mathcal J_3\mathcal C_1.
\end{align*}
The last assertion is \eqref{eq:stable-projector}.
\end{proof}

The identity endomorphism in this rational transfer system is the stable projector, not the identity of the full $K$-group.  This distinction is essential: elliptic stable characters control the summand $S_G$, but contain no information that would reconstruct $U_G$ across different inner forms.

\section{The rank-one calculation}\label{sec:sl2}

Let
\[
  G=\SL(2,\R),
  \qquad K=T=\SO(2),
  \qquad G_c=\SU(2).
\]
Then
\[
  W_G\cong\Z/2\Z,
  \qquad W_K=1,
  \qquad m_G=2,
  \qquad q(G)=1.
\]
Let
\[
  z=e^\rho\in R(T)=\Z[z,z^{-1}],
  \qquad
  s=z+z^{-1}.
\]
The Weyl involution is the ring automorphism
\[
  \iota:R(T)\longrightarrow R(T),
  \qquad \iota(z)=z^{-1},
\]
and
\[
  R(\SU(2))=\Z[s].
\]
For this isotropy representation the spin double cover admits the standard lift, so the spin representation group used in Dirac induction may be identified with $R(T)$.  Let
\[
  D:R(T)\overset{\sim}{\longrightarrow}K_0(\Cr\SL(2,\R))
\]
be the resulting Dirac-induction isomorphism.  For $t\in T^\dagger$, the rank-one instance of the orbital-trace formula is
\begin{equation}\label{eq:sl2-orbital}
  \tau_t^G(Dv)
  =-\frac{v(t)}{z(t)-z(t)^{-1}},
  \qquad v\in R(T).
\end{equation}
This is the specialization of \cite[Theorem~3.2]{HochsWang2019Orbital}; the corresponding fixed-point character calculation is displayed in \cite[Section~3.5]{HochsWang2018Shelstad}.

\subsection{The character identity as a trigonometric formula}

Write
\[
  t_\theta=
  \begin{pmatrix}
    \cos\theta&-\sin\theta\\
    \sin\theta&\cos\theta
  \end{pmatrix}\in T,
  \qquad z(t_\theta)=e^{i\theta}.
\]
Let \(\lambda=n\rho\) with \(n\geq1\).  The compact representation
\(E_\lambda\) has highest weight \((n-1)\rho\), and the Weyl character formula
gives
\begin{equation}\label{eq:sl2-compact-character}
  \chi_{E_\lambda}(t_\theta)
  =\frac{e^{in\theta}-e^{-in\theta}}
        {e^{i\theta}-e^{-i\theta}}
  =\frac{\sin(n\theta)}{\sin\theta}.
\end{equation}
The two noncompact fixed-point contributions are
\begin{align}
  \tau_{t_\theta}^G\bigl(Q_G((n-1)\rho)\bigr)
  &=\frac{e^{in\theta}}{2i\sin\theta},\label{eq:sl2-positive-index}\\
  \tau_{t_\theta}^G\bigl(Q_G(-(n-1)\rho)\bigr)
  &=-\frac{e^{-in\theta}}{2i\sin\theta}.
    \label{eq:sl2-negative-index}
\end{align}
Their sum is \eqref{eq:sl2-compact-character}; this is the explicit rank-one
instance of the stable index identity, worked out in
\cite[Section~3.5]{HochsWang2018Shelstad}.  Since \(q(G)=1\), the natural
discrete-series classes satisfy
\([\pi_{\pm\lambda}]=-Q_G(\pm(\lambda-\rho))\).  Hence the same calculation
becomes Shelstad's signed character identity
\[
  -\Theta_{\pi_\lambda}(t_\theta)
  -\Theta_{\pi_{-\lambda}}(t_\theta)
  =\frac{\sin(n\theta)}{\sin\theta}.
\]
This formula explains analytically why two L-packet members produce one compact
character; the remaining calculation records the integral lattice generated by
those identities.

\subsection{The integral lattice calculation}

\begin{proposition}[The two lattices for $\SL(2,\R)$]\label{prop:sl2}
Under Dirac induction,
\begin{align}
  \mathcal C_G(Dv)
  &=\frac{\iota(v)-v}{z-z^{-1}},\label{eq:sl2-C}\\
  \mathcal J_G(p(s))
  &=-D\bigl((z-z^{-1})p(s)\bigr),\label{eq:sl2-J}\\
  U_G&=D\bigl(\Z[s]\bigr),\label{eq:sl2-U}\\
  S_G&=D\bigl((z-z^{-1})\Z[s]\bigr).\label{eq:sl2-S}
\end{align}
Consequently,
\begin{equation}\label{eq:sl2-quotient}
  \frac{K_0(\Cr\SL(2,\R))}{S_G\oplus U_G}
  \cong(\Z/2\Z)[s].
\end{equation}
\end{proposition}

\begin{proof}
Stable conjugacy sends $t$ to $t^{-1}$.  By \eqref{eq:sl2-orbital},
\begin{align*}
  \tau_t^{\st,G}(Dv)
  &=\tau_t^G(Dv)+\tau_{t^{-1}}^G(Dv)\\
  &=-\frac{v(t)}{z(t)-z(t)^{-1}}
    +\frac{\iota(v)(t)}{z(t)-z(t)^{-1}}.
\end{align*}
The quotient $(\iota(v)-v)/(z-z^{-1})$ is Weyl invariant and lies in $\Z[s]$, which proves \eqref{eq:sl2-C}.  Its kernel is the invariant subring $\Z[s]$, proving \eqref{eq:sl2-U}.

If $p(s)\in\Z[s]$, then \eqref{eq:sl2-orbital} gives
\[
  \tau_t^G\left(-D\bigl((z-z^{-1})p(s)\bigr)\right)=p(s)(t).
\]
The characterization of $\mathcal J_G$ proves \eqref{eq:sl2-J} and \eqref{eq:sl2-S}.

Finally,
\[
  \Z[z,z^{-1}]=\Z[s]\oplus z\Z[s],
  \qquad z-z^{-1}=2z-s.
\]
Hence
\[
  \Z[s]+(z-z^{-1})\Z[s]
  =\Z[s]+2z\Z[s],
\]
and the quotient is $(\Z/2\Z)[s]$.
\end{proof}

For a positive integral parameter $\lambda=n\rho$, the L-packet is
\[
  \Pi_\lambda(G)=\{\pi_\lambda,\pi_{-\lambda}\}.
\]
If $E_\lambda$ is the irreducible $\SU(2)$-representation of highest weight $\lambda-\rho$, then
\begin{equation}\label{eq:sl2-packet}
  \mathcal J_G(E_\lambda)
  =-[\pi_\lambda]-[\pi_{-\lambda}],
  \qquad
  \mathcal C_G\mathcal J_G(E_\lambda)=2E_\lambda.
\end{equation}
The calculation shows more than the existence of \(2\)-torsion.  Each irreducible compact character contributes one independent parity obstruction, exactly as predicted by \eqref{eq:defect-direct-sum}.

The first two compact characters make the formulas concrete.  For the trivial
representation and the standard two-dimensional representation of
\(\mathrm{SU}(2)\), whose characters are \(1\) and \(s=z+z^{-1}\),
\begin{align*}
  \mathcal J_G(1)&=-D(z-z^{-1}),\\
  \mathcal J_G(s)&=-D(z^2-z^{-2}).
\end{align*}
In both cases \(\mathcal C_G\mathcal J_G\) doubles the compact character.
The denominator can also be seen on a single ambient class.  Since
\(\mathcal C_G(Dz)=-1\),
\begin{equation}\label{eq:sl2-half-projector}
  P_G^{\mathrm{st}}(Dz)
  =\frac12\mathcal J_G(-1)
  =\frac12D(z-z^{-1}).
\end{equation}
The right-hand side is not integral in \(R(T)=\mathbb Z[z,z^{-1}]\).  Thus the
factor \(1/2\) is not an artifact of the proof: it is already forced by the
simplest Dirac-induction generator outside \(S_G\oplus U_G\).

\section{Symplectic inner forms}\label{sec:symplectic}

Consider the inner class with complex group $\Sp(2n,\C)$.  The real forms and their maximal compact subgroups used below are standard; see \cite[Chapter~X]{Helgason2001}.  The absolute Weyl group is
\[
  W(C_n)\cong(\Z/2\Z)^n\rtimes S_n,
  \qquad |W(C_n)|=2^n n!.
\]
The split, quaternionic, and compact real forms considered below are inner to one another; at the Lie-algebra level this is also reflected by the absence of nontrivial diagram automorphisms in type $C_n$ \cite[Chapter~X]{Helgason2001}.

\begin{proposition}[Symplectic L-packet multipliers]\label{prop:symplectic}
Let $p+q=n$.
\begin{enumerate}[label=\textup{(\roman*)}]
\item For $G=\Sp(2n,\R)$ and $K=U(n)$,
\begin{equation}\label{eq:split-data}
  m_G=2^n,
  \qquad
  q(G)=\frac{n(n+1)}2,
  \qquad
  \varepsilon_G=(-1)^{n(n+1)/2}.
\end{equation}
\item For $G=\Sp(p,q)$ and $K=\Sp(p)\times\Sp(q)$,
\begin{equation}\label{eq:quaternionic-data}
  m_G=\binom np,
  \qquad
  q(G)=2pq,
  \qquad
  \varepsilon_G=1.
\end{equation}
\end{enumerate}
\end{proposition}

\begin{proof}
For the split form, $W_K\cong S_n$, so
\[
  [W(C_n):W_K]=\frac{2^n n!}{n!}=2^n.
\]
Moreover,
\[
  \dim\Sp(2n,\R)-\dim U(n)
  =n(2n+1)-n^2=n(n+1),
\]
which gives the value of $q(G)$.

For $\Sp(p,q)$,
\[
  W_K=W(C_p)\times W(C_q),
\]
and therefore
\[
  [W(C_n):W_K]
  =\frac{2^n n!}{(2^pp!)(2^qq!)}
  =\binom np.
\]
The symmetric space has real dimension $4pq$, so $q(G)=2pq$.
\end{proof}

Let $E_\lambda$ be the compact $\Sp(n)$-representation attached to a regular parameter.  The L-packet lifts are therefore
\begin{align}
  \mathcal J_{\Sp(2n,\R)}(E_\lambda)
  &=(-1)^{n(n+1)/2}
    \sum_{\pi\in\Pi_\lambda(\Sp(2n,\R))}[\pi],
    \label{eq:split-packet}\\
  \mathcal J_{\Sp(p,q)}(E_\lambda)
  &=\sum_{\pi\in\Pi_\lambda(\Sp(p,q))}[\pi].
    \label{eq:quaternionic-packet}
\end{align}
The corresponding defect groups are
\begin{equation}\label{eq:symplectic-defects}
  R(\Sp(n))/2^nR(\Sp(n))
  \quad\text{and}\quad
  R(\Sp(n))/\binom np R(\Sp(n)).
\end{equation}
The first is purely $2$-primary.  The second reflects the arithmetic of the binomial coefficient.  In particular, $\Sp(1,2)$ has $m_G=3$, so its stable projector has a genuine $3$-primary integral obstruction.

\subsection{The three real forms in type \texorpdfstring{$C_2$}{C2}}

Type $C_2$ gives a particularly transparent example with three real forms: the regular L-packet sizes are $4$, $2$, and $1$ on the split, quaternionic, and compact forms, respectively.

\begin{table}[ht]
\centering
\caption{L-packet multiplicities and signs in type $C_2$.  The entries are obtained from the Weyl-group and dimension formulas in \cref{prop:symplectic}.}
\label{tab:C2}
\begin{tabular}{@{}llllr@{}}
\toprule
$G$ & $K$ & $m_G$ & $q(G)$ & $\varepsilon_G$\\
\midrule
$\Sp(4,\R)$ & $U(2)$ & $4$ & $3$ & $-1$\\
$\Sp(1,1)$ & $\Sp(1)\times\Sp(1)$ & $2$ & $2$ & $+1$\\
$\Sp(2)$ & $\Sp(2)$ & $1$ & $0$ & $+1$\\
\bottomrule
\end{tabular}
\end{table}

Write
\[
  \Pi_\lambda(\Sp(4,\R))=\{\pi_1,\ldots,\pi_4\},
  \qquad
  \Pi_\lambda(\Sp(1,1))=\{\pi'_1,\pi'_2\}.
\]
Then stable transfer identifies
\begin{equation}\label{eq:C2-transfer}
  -\sum_{i=1}^4[\pi_i]
  \quad\longleftrightarrow\quad
  \sum_{j=1}^2[\pi'_j]
  \quad\longleftrightarrow\quad
  [E_\lambda].
\end{equation}
At the same time,
\begin{align*}
  \mathcal C_{\Sp(4,\R)}
  \left(-\sum_{i=1}^4[\pi_i]\right)&=4[E_\lambda],\\
  \mathcal C_{\Sp(1,1)}
  \left(\sum_{j=1}^2[\pi'_j]\right)&=2[E_\lambda].
\end{align*}
Thus transfer is integral on the L-packet lattices even though the ambient stable projectors require different denominators.  This distinction is exactly what the short exact sequence records.

\subsection{Type \texorpdfstring{$C_3$}{C3}: the first odd obstruction}

For \(n=3\), the inner class contains the split form
\(\Sp(6,\mathbb R)\), the quaternionic form \(\Sp(1,2)\), and the compact
form \(\Sp(3)\).  The Weyl-group data are listed in \cref{tab:C3}.
\begin{table}[H]
\centering
\caption{L-packet multiplicities in type \(C_3\).}
\label{tab:C3}
\begin{tabular}{@{}llllr@{}}
\toprule
\(G\) & \(K\) & \(m_G\) & \(q(G)\) & \(\varepsilon_G\)\\
\midrule
\(\Sp(6,\mathbb R)\) & \(U(3)\) & \(8\) & \(6\) & \(+1\)\\
\(\Sp(1,2)\) & \(\Sp(1)\times\Sp(2)\) & \(3\) & \(4\) & \(+1\)\\
\(\Sp(3)\) & \(\Sp(3)\) & \(1\) & \(0\) & \(+1\)\\
\bottomrule
\end{tabular}
\end{table}
If the corresponding packets are written as
\(\{\pi_1,\ldots,\pi_8\}\) and
\(\{\sigma_1,\sigma_2,\sigma_3\}\), then stable transfer identifies
\[
  \sum_{i=1}^8[\pi_i]
  \quad\longleftrightarrow\quad
  \sum_{j=1}^3[\sigma_j]
  \quad\longleftrightarrow\quad
  [E_\lambda].
\]
The middle L-packet transfers integrally, but its stable descent is
\[
  \mathcal C_{\Sp(1,2)}
  \left(\sum_{j=1}^3[\sigma_j]\right)=3[E_\lambda].
\]
Thus a \(3\)-primary obstruction occurs before any large-rank phenomenon is
involved.  This example also shows why the binomial coefficient in
\eqref{eq:quaternionic-data} carries genuine arithmetic information rather
than merely recording the size of a Weyl group.

\section{Conclusion and perspective}\label{sec:conclusion}

The argument separates into a geometric construction and an elementary lattice calculation.  Stable orbital integration turns noncompact index classes into compact characters, while Shelstad's identity lifts a compact character to a signed L-packet.  Once these maps are placed in opposite directions, every Weyl coset is counted once.  The resulting L-packet multiplicity is precisely the cokernel of the two trace-selected lattices.

The main conclusion is therefore not merely that a stable summand exists over $\Q$.  The sequence
\[
  0\longrightarrow S_G\oplus U_G
  \longrightarrow K_0(\Cr G)
  \longrightarrow R(G_c)/m_GR(G_c)
  \longrightarrow0
\]
identifies where integral splitting fails and shows that the failure is supported precisely at the primes dividing $m_G=\chi(X^d)$.  The rank-one calculation verifies the defect at the lattice level, while the symplectic family shows that odd obstruction primes occur naturally.

The same organization suggests a test for broader transfer constructions.  A candidate correspondence should first identify a family of geometric traces that separates the relevant $K$-theory, and then compare those traces with stable or endoscopic characters.  Outside the elliptic discrete-series setting, ordinary orbital traces alone are generally insufficient; higher cyclic cocycles provide natural additional functionals, and Song and Tang develop higher orbital integrals in a form adapted to noncommutative geometry \cite{SongTang2025}.  On the spectral side, Huang's spinorial description of transfer factors indicates how the required weights might arise from a relative Dirac index \cite{Huang2021}.  Turning these two observations into an endoscopic $KK$-class remains a separate problem.  Any extension compatible with the inner-form discrete-series case would, in particular, have to recover the integral L-packet-level constraint isolated here.


\end{document}